\documentclass[12pt]{amsart}
\usepackage{graphicx}
\usepackage{amsmath,amssymb}
\usepackage{amsthm}
\usepackage{color}
\usepackage{xcolor}

\newtheorem{thm}{Theorem}[section]
\newtheorem{cor}[thm]{Corollary}
\newtheorem{lem}[thm]{Lemma}
\newtheorem{prop}[thm]{Proposition}

\theoremstyle{definition}
\newtheorem{defn}[thm]{Definition}
\theoremstyle{remark}
\newtheorem{rem}[thm]{Remark}
\theoremstyle{example}
\newtheorem{exa}[thm]{Example}
\theoremstyle{conjecture}

\numberwithin{equation}{section}

\newcommand{\eps}{\varepsilon}

\newcommand{\del}{\delta}

\def\dlim{\displaystyle\lim}

\newcommand{\B}{{\mathbb B}}
\newcommand{\D}{{\mathbb D}}

\newcommand{\C}{{\mathbb C}}

\newcommand{\N}{{\mathbb N}}

\newcommand{\calO}{{\mathcal O}}

\newcommand{\calU}{{\mathcal U}}

\makeatletter
\newcommand{\doublewidetilde}[1]{{%
  \mathpalette\double@widetilde{#1}%
}}
\newcommand{\double@widetilde}[2]{%
  \sbox\z@{$\m@th#1\widetilde{#2}$}%
  \ht\z@=.9\ht\z@
  \widetilde{\box\z@}%
}
\makeatother

\begin{document}

\title[ ]{Sets of uniqueness, weakly sufficient sets and sampling sets for weighted spaces of holomorphic functions in the unit ball}%

\author{Bingyang Hu$^{1}$, Le Hai Khoi$^{2}$}

\address{Bingyang Hu: Department of Mathematics, University of Wisconsin, Madison, WI 53706-1388, USA.}%
\email{bhu32@wisc.edu}

\address{Le Hai Khoi: Division of Mathematical Sciences, School of Physical and Mathematical Sciences, Nanyang Technological University (NTU), 637371 Singapore, Singapore}%
\email{lhkhoi@ntu.edu.sg}

\date{\today}%

\thanks{$^{1}$ Supported in part NSF grant DMS 1600458 and NSF grant 1500182.}
\thanks{$^{2}$ Supported in part by MOE's AcRF Tier 1 grant M4011724.110 (RG128/16)}

\maketitle

\begin{abstract}
We consider inductive limits of weighted spaces of holomorphic functions in the unit ball of $\C^n$. The relationship between sets of uniqueness, weakly sufficient sets and sampling sets in these spaces is studied. In particular, the equivalence of these sets, under general conditions of the weights, is obtained.
\end{abstract}

\section{Introduction}

Let $\B$ be the unit ball in $\C^n$ and $\calO(\B)$ be the collection of holomorphic functions on $\B$ with the usual compact-open topology. A function $\varphi$ defined on $[0, 1)$ is said to be a \emph{weight} if it takes the values in $(1, +\infty)$. We define the associated weighted space $A_\varphi$ as
\[
A_\varphi:=\left\{ f \in \calO(\B): \|f\|_\varphi:=\sup_{z \in \B} \frac{|f(z)|}{\varphi(|z|)}<+\infty \right\}.
\]
In this paper, we are interested in the case where the function space is defined by the inner inductive limit of $A_\varphi$'s. More precisely, let $\Phi=\left( \varphi_p \right)_{p=1}^\infty$ denote an increasing sequence of weights. For simplicity, we use $\|f\|_p$ instead of $\|f\|_{\varphi_p}$, and $A_\varphi^{-p}$ instead of $A_{\varphi_p}$. Then clearly $A_{\varphi}^{-p} \hookrightarrow A_\varphi^{-(p+1)}$ (here $\hookrightarrow$ means a continuous embedding), and hence, we let
\[
A^{-\infty}_\Phi:= \bigcup_{p \ge 1} A_{\varphi}^{-p}
\]
with the topology induced by the \emph{inner inductive limit} of $A_\varphi^{-p}$, namely
\[
(A^{-\infty}_\Phi, \tau)=\textrm{lim ind} A_{\varphi}^{-p}.
\]
Now let $S$ be a subset set of $\B$. Define
\[
A_{\varphi}^{-p, S}:=\left\{ f \in A_{\Phi}^{-\infty}: \|f\|_{p, S}=\sup_{z \in S} \frac{|f(z)|}{\varphi_p(|z|)}<\infty \right\}.
\]
Notice that the inclusion relations $A_\varphi^{-p} \hookrightarrow A_\varphi^{-p, S}  \hookrightarrow A_\Phi^{-\infty}$ always hold.  Hence, it follows immediately that
\[
A_\Phi^{-\infty}=\bigcup_{p \ge 1} A_{\varphi}^{-p} \subseteq \bigcup_{p \ge 1} A_\varphi^{-p, S} \subseteq A_{\Phi}^{-\infty},
\]
which implies
\[
A_\Phi^{-\infty}=\bigcup_{p \ge 1} A_\varphi^{-p}=\bigcup_{p \ge 1} A_\varphi^{-p, S}.
\]
Therefore, we can endow $A_\Phi^{-\infty}$ with another weaker inner inductive limit topology of seminormed spaces $A_\varphi^{-p, S}$:
\[
(A_\Phi^{-\infty}, \tau_S):=\textrm{lim ind} A_{\varphi}^{-p, S}.
\]

 We note that this type of spaces appeared as duals to Fr\'echet-Schwartz (FS) spaces and play an important role in the study of representing holomorphic functions in series of simpler functions, such as exponential functions, or rational functions, which have many applications in functional equations and approximations of functions (see, e.g., \cite{AK1, AK2, AKN, Bar, Kor, Sch, Tay} and references therein).

\medskip
We are interested in the following three special cases of $S$.

\begin{defn}
A set $S \subseteq \B$ is said to be \emph{weakly sufficient} for the space $A_\Phi^{-\infty}$ if two topologies $\tau$ and $\tau_S$ are equivalent.
\end{defn}

\begin{defn}
A set $S \subseteq \B$ is called \emph{a set of uniqueness} for $A_\Phi^{-\infty}$ if $f \in A_\Phi^{-\infty}$ and $f(z)=0$ for all $z \in S$ imply that $f=0$.
\end{defn}

Furthermore, let a continuous function $\psi: [0, 1)\to (0, \infty)$ satisfy $\dlim_{r \to 1} \psi(r)=\infty$. For $f \in A_\Phi^{-\infty}$, we put
\[
T_{f, \psi}=\limsup_{|z| \to 1} \frac{\log |f(z)|}{\psi(|z|)},
\]
and for $S \subset \B$,
\[
T_{f, \psi, S}=\limsup_{|z| \to 1, z \in S} \frac{\log |f(z)|}{\psi(|z|)}.
\]

\begin{defn}
A set $S \subseteq \B$ is called \emph{a $\psi$-sampling set} for $A_\Phi^{-\infty}$, if $T_{f, \psi}=T_{f, \psi, S}$ for every $f \in A_{\Phi}^{-\infty}$.
\end{defn}

A \emph{classical example} of $\Phi$ and $\psi$ is
\[
\Phi=\left\{\varphi_p(r)=(1-r)^{-p}, p\in\N\right\}, \ \psi(r)=|\log(1-r)|.
\]
Such a choice of $\Phi$ and $\psi$ is motivated by the easy fact that $\textrm{dist}(z, \partial \B)=1-|z|, z \in \B$. Moreover, in this case, the space $A_\Phi^{-\infty}$ becomes the well-known function space $A^{-\infty}$, which consists of all the holomorphic functions on $\B$ with polynomial growth.

The property of weakly sufficient sets, sets of uniqueness and sampling sets for $A^{-\infty}$ was carefully studied in \cite{KT} for the case $n=1$ and in \cite{AA1, LHK} for the higher dimension case. More precisely, the the following results have been obtained in these papers.

\begin{thm}
Let $S$ be a subset of the unit disc $\B$ (respectively, $\D$). Consider the following assertions
\begin{enumerate}
\item[(i)] $S$ is classical sampling for $A^{-\infty}(\B)$ (respectively, $\D$);
\item[(ii)] $S$ is weakly sufficient for $A^{-\infty}(\B)$ (respectively, $\D$);
\item[(iii)] $S$ is a set of uniqueness for $A^{-\infty}(\B)$ (respectively, $\D$).
\end{enumerate}
Then $(i) \implies (ii)  \implies (iii)$.
\end{thm}

It should be noted that the inverse implications, in general, are not true. In \cite{KT}, two counter-examples are provided to show  a failure of both inverse implications. Nevertheless, in \cite{LHK} it was showed, for the classical $\psi$, that $(iii) \implies (ii)$, with an additional assumption.

A careful study of proofs of the results above led us to the thought that $(ii)$ does not imply $(i)$ due to `` too independent'' growths of $\Phi$ and $\psi$, while $(iii)$ does not imply $(ii)$ because the assumption of $S$ to be a set of uniqueness is not strong enough to ensure an inclusion to become a continuous embedding.

So a question to ask is for what $\Phi$ and $\psi$, we can have
\[
(i) \Longleftrightarrow (ii)  \Longleftrightarrow (iii) \  ?
\]

In the recent paper \cite{AA}, the affirmative answer was given for the case $\varphi_p(r)=e^{p g(r)}, 0<p<\infty$ and $\psi(r)=g(r)$, where $g$ is a so-called \emph{everywhere quasi-canonical weight} (see, \cite[Section 4]{AA}).

The aim of this paper is to solve the question above for a general $\Phi$ and $\psi$, with the ``minimal" assumptions imposed on them.

The structure of this paper is as follows. In Section 2, we show that $(ii) \implies (iii)$ always. Then we introduce conditions $(C_1)-(C_2)$ for $\Phi$ (for which the classical weights are satisfied) and show that with the same additional assumption as for the classical case, $(iii) \implies (ii)$ (Theorem 2.5). The proof follows the scheme in \cite{LHK}. Section 3 deals with the equivalence $(ii) \Longleftrightarrow (i)$.  We introduce conditions $(C_3)-(C_5)$ which ``relate'' growths of $\Phi$ and $\psi$" and, together with $(C_1), (C_2)$, allow us to establish each of implications $(i)\implies (ii)$ and $(ii) \implies (i)$.

\section{Weakly sufficient sets and sets of uniqueness}

The following characterization of weakly sufficient sets is needed in the sequel.

\begin{lem} {\cite[Proposition 2.2]{LHK2}}  \label{lem01}
For the space $A_\Phi^{-\infty}$, the following statements are equivalent:
\begin{enumerate}
\item [(a)] $S$ is weakly sufficient set.
\item [(b)] For any $p \ge 1$, there exists $m=m(p)$, such that
\[
A_\varphi^{-p,S} \hookrightarrow A_\varphi^{-m},
\]
i.e., for any $p \ge 1$, there exist $m=m(p)$ and $C=C(p)$, such that
\[
\|f\|_m \le C\|f\|_{p, S}, \ \text{for all}\ f \in A_{\varphi}^{-p, S}.
\]
\item [(c)] For any $p \ge 1$, there exist $m=m(p)$ and $C=C(p)$, such that
\[
\|f\|_m \le C\|f\|_{p, S}, \ \text{for all}\ f \in A_\Phi^{-\infty}.
\]
\end{enumerate}
\end{lem}

The following proposition is an easy modification of \cite[Proposition 3.1]{LHK}, with replacing the factor $(1-|z|)^p$ by $\varphi_p(|z|)^{-1}$, and hence we omit the proof here.

\begin{prop}
Any weakly sufficient set for $A_{\Phi}^{-\infty}$ is a set of uniqueness for this space.
\end{prop}

\begin{rem} \label{rem01}
We give an example of a set of uniqueness which is not a weakly sufficient set. Let $S=\{z \in \B: |z| \le 1/2\}$, $\varphi_p(r)=(1-r)^{-p}, p \in \N$, and the test functions
\[
f_k(z)=3^{-k} (1-z_1)^{-k}, k \in \N,
\]
where $z=(z_1, \dots, z_n) \in \B$. Clearly, $S$ is a set of uniqueness. Moreover, for a fixed $p_0>0$, a direct calculation shows $\|f_k\|_{p_0, S}$ is uniformly bounded in $k$, but for any $p>0$, $\|f_k\|_p=\infty$ for $k$ sufficiently large, which implies $S$ fails to be a weakly sufficient set.
\end{rem}

Now we proceed the other direction, and we are inspired by the idea from \cite{LHK}. For a given $\Phi=(\varphi_p)_{p=1}^\infty$, we consider the following conditions
\begin{equation} \tag{$\mathbf{C_1}$}
\lim_{r \to 1} \frac{\varphi_{p+1}(r)}{\varphi_p(r)}=\infty, \ \text{for all}\ p \ge 1,
\end{equation}
and
\begin{equation} \tag{$\mathbf{C_2}$}
\text{for all}\ p \ge 1, \ \frac{\varphi_{p+1}}{\varphi_p} \ \text{is bounded on any compact subset of}\ [0,1).
\end{equation}
Clearly, if all $\varphi_p$ are \emph{continuous} on $[0, 1)$, then $\mathbf{(C_2)}$ is satisfied.

\begin{exa}
We give some examples to show that these conditions are independent each of other.
\begin{enumerate}
\item The classical weights $\Phi=( (1-r)^{-p} )_{p=1}^\infty$ satisfy both $\mathbf{(C_1)}$ and $\mathbf{(C_2)}$.
\item For any $a>2$, the weights $\Phi=\left( a+r-\frac{1}{p} \right)_{p=1}^\infty$ satisfy $\mathbf{(C_2)}$ but not $\mathbf{(C_1)}$.
\item For each $p \in \N$, let
\[
\varphi_p(r)=\begin{cases}
\left(\frac{1}{2}-r \right)^{-p}, & r \in \left[0, \frac{1}{2}\right),\\
p^p, & r \in \left[\frac{1}{2}, 1 \right).
\end{cases}
\]
Then the weights $\Phi=(\varphi_p)_{p=1}^\infty$ satisfy $\mathbf{(C_1)}$ but not $\mathbf{(C_2)}$.
\item For $p \in \N$, consider
\[
\varphi_p(r)=\begin{cases}
\left(\frac{1}{2}-r \right)^{-p}, & r \in \left[0, \frac{1}{2}\right),\\
1, & r \in \left[\frac{1}{2}, 1 \right).
\end{cases}
\]
Then the weights $\Phi=(\varphi_p)_{p=1}^\infty$ do not satisfy neither $\mathbf{(C_1)}$ nor $\mathbf{(C_2)}$.
\end{enumerate}
\end{exa}

The following is the main result in this section.

\begin{thm} \label{thm01}
Suppose $\Phi$ satisfies conditions ${\bf (C_1)}$ and ${\bf (C_2)}$. Then $S \subset \B$ is a weakly sufficient set for $A_{\Phi}^{-\infty}$ if and only if
\begin{enumerate}
\item[(a)] $S$ is a set of uniqueness for $A_\Phi^{-\infty}$.
\item[(b)] For any $p \ge 1$, there exists an $m=m(p)$, such that $A_\varphi^{-p, S} \subset A_\varphi^{-m}$.
\end{enumerate}
\end{thm}

This result shows that under conditions $\mathbf{(C_1)}-\mathbf{(C_2)}$, a set-inclusion (statement (b)) implies a continuous embedding.

\begin{proof}
Since the necessity is obvious, it suffices to show the sufficiency.  For any $p \ge 1$ given, there exists by (b), an $m=m(p)$, such that $A_\varphi^{-p, S} \subset A_\varphi^{-m}$. Without loss of generality, we may assume that $m \ge p$.

Now we take and fix an integer $q>m+2$ and let us denote $U_{p, S}$ the unit ball in $A_\varphi^{-p, S}$ and $U_q$ the unit ball in $A_\varphi^{-q}$. Since
\[
\sup_{f \in U_{p, S}} \|f\|_q \le \sup_{f \in U_{p, S} \backslash U_q} \|f\|_q+\sup_{f \in U_q} \|f\|_q,
\]
it suffices for us to show that $U:=U_{p, S} \backslash U_q$ is bounded in the space $A_\varphi^{-q}$.

At this point we pause and prove the following result for the norm $\|f\|_q$.

\begin{lem} \label{lem02}
Let $\ell>m$ be some integer. Then there exists some $s=s(\ell) \ge 1, s \in \N$, such that the $\|f\|_{\ell}$ is attained on the compact set $K_s:=\left\{z \in \B: |z| \le \frac{s}{s+1} \right\}$ for all $f \in U$.
\end{lem}

\begin{proof}
We prove the result by contradiction. Namely, for any $s \in \N$, there exists a $f_s \in U$, such that
\[
\sup_{K_s} \frac{|f_s(z)|}{\varphi_{\ell}(|z|)}<\|f_s\|_{\ell},
\]
from which it follows that
\begin{equation} \label{eq201}
\sup_{\B \backslash K_s} \frac{|f_s(z)|}{\varphi_{\ell}(|z|)}=\|f_s\|_{\ell}.
\end{equation}
The proof will complete if we can construct a function $u \in A_\varphi^{-p, S}$ but $u \notin A_\varphi^{-({\ell}-1)}$, which will contradict to our early assumption $A_\varphi^{-p, S} \subset A_\varphi^{-m} \subset A_\varphi^{-({\ell}-1)}$ (as $m \le {\ell}-1$). Moreover, this function  is constructed in the form of a series
\[
u(z)=\sum_{k=1}^\infty c_ku_k(z)
\]
where $u_k \in \calU:=\left\{ \frac{f}{\|f\|_{\ell}}; f \in U \right\}$ will be defined inductively as follows.

Take an arbitrary $u_1 \in \calU$. If we have $u_1, u_2, \dots, u_{t-1} \in \calU (t \ge 2)$, then $u_t$ is determined in the following way:

\medskip
\textit{Step I:} By condition  ${\bf (C_1)}$, we can choose $s_t \in \N$ large enough so that
\begin{equation} \label{eq202}
\frac{\varphi_{\ell}(|z|)}{\varphi_{{\ell}-1}(|z|)} \ge 3 \cdot 4^{t+1} \sum_{k=1}^{t-1} \|u_k\|_{{\ell}-1}, \quad \text{for all} z \in \B \backslash K_{s_t}.
\end{equation}
Note that here the quantity $\|u_k\|_{{\ell}-1}$ is well-defined. Indeed, for any $f \in U=U_{p, S} \backslash U_{\ell}$, we have $f \in A_\varphi^{-p, S} \subset A_\varphi^{-m} \subset A_\varphi^{-({\ell}-1)} \subset A_\varphi^{-{\ell}}$, which implies the desired claim.

\medskip
\textit{Step II:} For the $K_{s_t}$ chosen above, there is an $f_t \in U$ such that \eqref{eq201} holds, and we define $
u_t(z)=\frac{f_t(z)}{\|f_t\|_{{\ell}}} \in \calU$.

\medskip
Thus from the above two steps, a sequence $\{u_k\} \subset \calU$ is defined. Taking $c_k=4^{-k}, k=1, 2, \dots$, we have
\[
u(z)=\sum_{k=1}^\infty 4^{-k}u_k(z),
\]
which implies
\[
\|u\|_{\ell} \le \sum_{k=1}^\infty \frac{\|u_k\|_{\ell}}{4^k}=\sum_{k=1}^\infty \frac{1}{4^k}<\infty.
\]
This shows that $u \in A_\varphi^{-{\ell}}$.

\medskip
\textit{Claim I: $u \in A_\varphi^{-p, S}$.}

\medskip
Note that for any $f \in U=U_{p, S} \backslash U_{\ell}$, we have $\|f\|_{p, S} \le 1$ and $\|f\|_{\ell}>1$. Hence, for any $g \in \calU$, it always holds that $\|g\|_{p, S}<1$, which implies
$$
\|u\|_{p, S} \le \sum_{k=1}^\infty \frac{1}{4^k} \|u_k\|_{p, S} \le \sum_{k=1}^\infty \frac{1}{4^k}<\infty,
$$
which implies the first claim.

\medskip
\textit{Claim II: $u \notin A_\varphi^{-({\ell}-1)}$.}

\medskip
Take and fix any $t \in \N$. Then from the Step II above, there exists an $z_t \in \B \backslash K_{s_t}$ such that
\begin{equation} \label{eq203}
\frac{|f_t(z_t)|}{\varphi_{\ell}(|z_t|)} \ge \frac{\|f_t\|_{\ell}}{2}.
\end{equation}
Thus
\begin{eqnarray*}
|u(z_t)|%
&=& \left| \sum_{k=1}^\infty \frac{u_k(z_t)}{4^k} \right| \ge \left| \frac{u_t(z_t)}{4^t} \right|-  \left| \sum_{k \neq t} \frac{u_k(z_t)}{4^k} \right| \\
&\ge&  \left| \frac{u_t(z_t)}{4^t} \right|-  \sum_{k \neq t} \frac{|u_k(z_t)|}{4^k} \\
&=&  \left| \frac{u_t(z_t)}{4^t} \right|-  \sum_{k=1}^{t-1} \frac{|u_k(z_t)|}{4^k}- \sum_{k=t+1}^{\infty} \frac{|u_k(z_t)|}{4^k}
= J_0-J_1-J_2.
\end{eqnarray*}

\medskip
\noindent\textit{Estimate of $J_0$}. Applying \eqref{eq203}, we have
\[
J_0 \ge \frac{1}{4^t} \cdot \frac{\varphi_{\ell}(|z_t|)}{2}.
\]

\medskip
 \noindent\textit{Estimate of $J_1$}. For each $k \in\N$, we have
\[
|u_k(z_t)|= \frac{|u_k(z_t)|}{\varphi_{{\ell}-1}(|z_t|)} \cdot \varphi_{{\ell}-1}(|z_t|) \le \|u_k\|_{{\ell}-1} \varphi_{{\ell}-1}(|z_t|),
\]
and hence
\[
J_1 \le \varphi_{{\ell}-1}(|z_t|) \cdot \sum_{k=1}^{t-1} \frac{\|u_k\|_{{\ell}-1}}{4^k}.
\]

\medskip
 \noindent\textit{Estimate of $J_2$}. Similarly, for each $k \in \N$, we have
\[
|u_k(z_t)|= \frac{|u_k(z_t)|}{\varphi_{\ell}(|z_t|)} \cdot \varphi_{\ell}(|z_t|) \le \varphi_{\ell}(|z_t|),
\]
which gives
\[
J_2 \le \varphi_{\ell}(|z_t|) \cdot  \sum_{k=t+1}^\infty \frac{1}{4^k}=\frac{\varphi_{\ell}(|z_t|)}{3 \cdot 4^t}.
\]
Combining the three estimates above yields
\begin{eqnarray*}
|u(z_t)|%
&\ge& \left( \frac{1}{2} \cdot \frac{1}{4^t}- \frac{1}{3} \cdot \frac{1}{4^t} \right)  \cdot \varphi_{\ell}(|z_t|)-\varphi_{{\ell}-1}(|z_t|) \cdot \sum_{k=1}^{t-1} \frac{\|u_k\|_{{\ell}-1}}{4^k} \\
&=& \frac{ \varphi_{{\ell}-1}(|z_t|)}{6 \cdot 4^t} \cdot \frac{\varphi_{\ell}(|z_t|)}{\varphi_{{\ell}-1}(|z_t|)}-\varphi_{{\ell}-1}(|z_t|) \cdot \sum_{k=1}^{t-1} \frac{\|u_k\|_{{\ell}-1}}{4^k}\\
&\ge& \frac{ \varphi_{{\ell}-1}(|z_t|)}{6 \cdot 4^t} \cdot 3 \cdot 4^{t+1} \sum_{k=1}^{t-1} \|u_k\|_{{\ell}-1}-\varphi_{{\ell}-1}(|z_t|) \cdot \sum_{k=1}^{t-1} \frac{\|u_k\|_{{\ell}-1}}{4^k}\\
&& \quad \quad (\textrm{by \eqref{eq202}}) \\
&\ge& \varphi_{{\ell}-1}(|z_t|) \cdot \sum_{k=1}^{t-1} \|u_k\|_{{\ell}-1}.
\end{eqnarray*}
Noting the easy fact that $\|u_k\|_{{\ell}-1} \ge \|u_k\|_{\ell}=1, \text{for all} k \ge 1$, we have for any $t \ge 2$,
$$
\|u\|_{{\ell}-1}= \sup_{z \in \B} \frac{|u(z)|}{\varphi_{{\ell}-1}(|z|)} \ge \frac{|u(z_t)|}{\varphi_{{\ell}-1}(|z_t|)} \ge  \sum_{k=1}^{t-1} \|u_k\|_{{\ell}-1}>t-1,
$$
which implies $\|u\|_{{\ell}-1}=\infty$ and therefore $u \notin A_{\varphi}^{-({\ell}-1)}$.

Thus, there exists some $s$, such that the $\|f\|_{\ell}$ is attained on $K_s$ uniformly for all $f \in U$.
\end{proof}

\bigskip
Now return back to the proof of Theorem \ref{thm01}. By Lemma \ref{lem02}, we take and fix some $s \in \N$ so that $\|f\|_{q-1}$ is attained on $K_s$ for all $f \in U$ (note that $q>m+2$, and hence $q-1>m$).

By condition ${\bf (C_2)}$, there exists $M=M(q, s)>0$, such that
\[
\frac{\varphi_q(|z|)}{\varphi_{q-1}(|z|)} \le M, \quad z \in K_s
\]
and hence for any $f \in U$,
\begin{eqnarray*}
\|f\|_{q-1}%
&=& \sup_{z \in \B} \frac{|f(z)|}{\varphi_{q-1}(|z|)}=\sup_{z \in K_s} \frac{|f(z)|}{\varphi_{q-1}(|z|)}= \sup_{z \in K_s} \frac{|f(z)|}{\varphi_q(|z|)} \cdot \frac{\varphi_q(|z|)}{\varphi_{q-1}(|z|)}\\
&\le& M  \sup_{z \in K_s} \frac{|f(z)|}{\varphi_q(|z|)} \le M\|f\|_q,
\end{eqnarray*}
which shows that the set
\[
\left\{\frac{f}{\|f\|_q} \right\}_{f \in U}
\]
is bounded in $A_\varphi^{-(q-1)}$. An easy application of Montel's theorem and condition  ${\bf (C_1)}$ imply that the identity map $i: A_\varphi^{-(q-1)}\to A_\varphi^{-q}$ is compact, and hence the set $\left\{\frac{f}{\|f\|_q} \right\}_{f \in U}$ is relatively compact in $A_\varphi^{-q}$.

Recall that we have to show $U$ is bounded in $A_\varphi^{-q}$, namely
\begin{equation} \label{eq204}
\sup_{f \in U} \|f\|_q<\infty.
\end{equation}
We prove it by contradiction. Assume \eqref{eq204} does not hold. Then we can take a sequence $\{f_k\} \subset U$ such that
\begin{equation} \label{eq205}
\lim_{k \to \infty} \|f_k\|_{q}=\infty.
\end{equation}
By the remark above, we have the set $\left\{\frac{f_k}{\|f_k\|_q} \right\}_k$ is sequentially compact, namely, there is a sequence
\begin{equation} \label{eq206}
g_\ell:=\frac{f_{k_l}}{\|f_{k_l}\|_q} \to g_0
\end{equation}
as $\ell \to \infty$ in the sense of $A_{\varphi}^{-q}$. Clearly, $g_0 \in A_\varphi^{-q}$ with $\|g_0\|_q=1$. This implies $\|g_0\|_{q, S}=C>0$, since $S$ is a set of uniqueness. Moreover, by \eqref{eq206}, one can also see that $\lim\limits_{\ell \to \infty} \|g_\ell-g_0\|_{q, S}=0$, which implies for sufficient large $\ell$, we have
\[
\|g_\ell\|_{q, S} \ge \frac{\|g_0\|_{q, S}}{2}=\frac{C}{2}.
\]
On the other hand,
\[
\|g_\ell\|_{q, S}=\left\| \frac{f_{k_l}}{\|f_{k_l}\|_q}\right\|_{q, S}=\frac{\|f_{k_l}\|_{q, S}}{\|f_{k_l}\|_q} \le \frac{1}{\|f_{k_l}\|_q},
\]
where in the last inequality, we use the fact that $\|f_{k_l}\|_{q, S} \le \|f_{k_l}\|_{p, S} \le 1$. Combining the above estimates yields 
\[
\|f_{k_l}\|_q \le \frac{2}{C},
\] 
which contradicts to the assumption \eqref{eq205}. The proof is complete.
\end{proof}

\section{Sampling sets and weakly sufficient sets}

In this section, we study the relation between sampling sets and weakly sufficient sets, in particular we investigate under what conditions weakly sufficient sets are sampling, and vice versa.

As said in Introduction, weakly sufficient sets fail to be sampling sets because the growths of $\Phi$ and $\psi$ are too ``independent''. Motivated by this, we introduce some conditions which ``relate'' growths of the pair $(\Phi, \psi)$.
\begin{equation} \tag{$\mathbf{C_3}$}
\text{for each $p\in\N$, there exists}\ \lim_{r \to 1} \frac{\log \varphi_p(r)}{\psi(r)}:=\alpha_p,
\end{equation}
\begin{equation} \tag{$\mathbf{C_4}$}
(\alpha_p)\ \text{is strictly increasing},
\end{equation}
\begin{equation} \tag{$\mathbf{C_5}$}
(\alpha_p)\ \text{is bounded}.
\end{equation}

In case $\mathbf{(C_3)}-\mathbf{(C_5)}$ are satisfied, we denote
\[
\alpha:=\lim_{p\to\infty} \alpha_p,
\]
and for each $\xi\in\B, t\in[0,\alpha)$, we put
\[
R(\psi, \xi, t):=\left\{g \in \calO(\B): |g(\xi)|=e^{t\psi(|\xi|)}\right\}.
\]

Finally, we introduce the following condition for $(\Phi, \psi)$:
\begin{equation} \tag{$\mathbf{C_6}$}
\text{there exists $M>0$, such that for each $\xi \in \B$ and $t \in [0,\alpha)$,}
\end{equation}
\[
\text{there is a $f_{\xi,t} \in R(\psi,\xi,t)$ for which $|f_{\xi, t}(z)| \le M e^{t\psi(|z|)}$, for all $z \in \B$.}
\]

\medskip
Now we can state the main result in this section.

\begin{thm} \label{thm02}
\begin{enumerate}
\item [(1)] Let $\Phi$ and $\psi$ satisfy conditions $\mathbf{(C_1)}-\mathbf{(C_4)}$. Then every $\psi$-sampling set is a weakly sufficient set.
\item [(2)] Let $\Phi$ and $\psi$ satisfy conditions $\mathbf{(C_2)}-\mathbf{(C_6)}$. Let further, $\varphi_1$ is bounded on any compact subset of $[0, 1)$. Then every weakly sufficient set is a $\psi$-sampling set.
\end{enumerate}
\end{thm}

\begin{rem}\label{r30}
Let us consider the classical case, that is $\varphi_p(r)=(1-r)^{-p}, \ p\in\N$ and $\psi(r)=|\log(1-r)|$.

First, it is easy to check that conditions $\mathbf{(C_1)}-\mathbf{(C_4)}$ are satisfied. So Theorem \ref{thm02} (1) contains the classical result of the implication ``sampling sets $\implies$ weakly sufficient sets'' as a particular case.

On the other hand, the condition $\mathbf{(C_5)}$ is not satisfied, as $\alpha(p)=p$, is unbounded. That is, the classical pair $(\Phi,\psi)$ does not satisfy the assumption in Theorem \ref{thm02} (2).
\end{rem}

\begin{rem}
Concerning condition $\mathbf{(C_6)}$, we give some motivations for its appearance.

First, a typical example that makes condition $\mathbf{(C_6)}$ non-trivial is
\begin{equation} \label{e01}
\psi(r)=-\log (1-r^2), 0<r<1.
\end{equation}
Actually, we can prove a slightly stronger assertion for such a choice of $\psi$: for any $\xi \in \B$ and $t>0$, there exists a $f_{\xi, t} \in R(\psi, \xi, t)$, such that
\[
|f_{\xi, t}(z)| \le e^{t\psi(|z|)}.
\]
Indeed, this follows from
\begin{equation} \label{e02}
f_{\xi, t}(z)=\frac{(1-|\xi|^2)^t}{(1-\langle z, \xi \rangle)^{2t}}
\end{equation}
and an easy fact
\[
1-|\Phi_{\xi}(z)|^2=\frac{(1-|\xi|^2)(1-|z|^2)}{|1- \langle \xi, z \rangle|^2}, \quad \xi, z \in \B
\]
(see, \cite[Lemma 1.2]{Zhu}), where $\Phi_\xi$ is the involutive automorphism in $\B$ associated to $\xi$.

The example above works for all $t>0$, that is why for $\psi$ chosen in \eqref{e01} we get a slightly stronger result, which is independent of the choice of $\alpha$ in $\mathbf{(C_6)}$. So is the collection $\Phi$.

This example leads us to the following.
\begin{enumerate}
\item [1.] There are many ``$\psi$" that satisfy $\mathbf{(C_6)}$. One can take, say $\psi(r)=-3\log(1-r^2)$ or $\psi(r)=-\log 2(1-r)$, $0<r<1$.
\item [2.] The classical case which is considered in Remark \ref{r30} satisfies condition $\mathbf{(C_6)}$.
\item [3.] There is an alternative way to think about $\mathbf{(C_6)}$. More precisely, for each $t>0$, we define the Banach space
$$
A_{t, \psi}:=\left\{ f \in \calO(\B): \sup_{z \in \B} |f(z)| e^{-t \psi(|z|)}<\infty\right\}.
$$
Then $\mathbf{(C_6)}$ can be restated as follows: there exists $M>0$, such that for any $\xi \in \B$ and $t \in [0, \alpha)$, there is $f_{\xi,t} \in \B_{t, \psi}(M)$, satisfying $|f_{\xi, t}(\xi)|=e^{t \psi(|\xi|)}$, where $\B_{t,\psi}(M)$ is the ball of radius $M$ in $A_{t, \psi}$.

Note also that the space $A_{t, \psi}$ can be thought as of the ``growth space" of certain space of holomorphic functions on the unit ball. For instance, in our above example, $A_{t, \xi}$ is the Bergman-type spaces $A^{-t}$, which is the ``growth space" of   the classical Bergman space $A_{n+1-t}^1$.  We refer the reader to \cite{Zhu} for detailed information about these spaces.
\end{enumerate}
\end{rem}

\medskip
\noindent\textit{Proof of Theorem \ref{thm02}}.

\subsection{Proof of (1)}

Let $S \subset \B$ be a $\psi$-sampling set for $A_\Phi^{-\infty}$. Since conditions $\mathbf{(C_1)-(C_2)}$ are satisfied, we prove that two conditions in Theorem \ref{thm01} are true.

\medskip
\textit{ (a). $S$ is a set of uniqueness}.

\medskip
Suppose $f \in A_\Phi^{-\infty}$ and $f(z)=0$ for all $z \in S$. Then since $S$ is a $\psi$-sampling set, we have
\[
T_{f, \psi}=T_{f, \psi, S}=\limsup_{|z| \to 1, z \in S} \frac{\log |f(z)|}{\psi(|z|)}=-\infty.
\]
In particular, $\displaystyle T_{f, \psi}=\limsup_{|z| \to 1} \frac{\log |f(z)|}{\psi(|z|)}<-1$ which implies that there exists a $\del>0$, such that $\log|f(z)|\le -\psi(|z|)$ for all $z$ with $1-\del<|z|<1$.

Take an arbitrary $w \in \B$. By the Maximum modulus principle, we have
\[
|f(w)| \le \sup_{|z|=r} |f(z)| \le e^{-\psi(r)},\ \text{ for all $r\in(\max\{1-\del,|w|\},1)$}.
\]
Letting $r \to 1^{-}$, since $\dlim_{r\to1}|\psi(r)|=\infty$, we get $f(w)=0$.

\medskip
\textit{(b). For any $p \in \N$, there exists an $m=m(p)$, such that $A_\varphi^{-p, S} \subset A_{\varphi}^{-m}$}.
\medskip

Let $p\in\N$. For every $f \in A_\varphi^{-p, S}$, since
\[
\sup_{z \in S} \frac{|f(z)|}{\varphi_p(|z|)}<\infty,
\]
there exists $C>0$ such that
\[
|f(z)|<C\varphi_p(|z|),\ \text{for all $z \in S$},
\]
which gives
\[
\frac{\log|f(z)|}{\psi(|z|)}<\frac{\log C}{\psi(|z|)} + \frac{\log \varphi_p(|z|)}{\psi(|z|)},\ \text{for all $z \in S$}.
\]

Since $S$ is a $\psi$-sampling set,  by condition $\mathbf{(C_3)}$, we have
\begin{eqnarray*}
T_{f, \psi}%
&=& T_{f, \psi, S} = \limsup_{|z| \to 1, z \in S} \frac{\log|f(z)|}{\psi(|z|)} \le
\displaystyle\limsup_{|z| \to 1, z \in S} \left(\frac{\log C}{\psi(|z|)} + \frac{\log \varphi_p(|z|)}{\psi(|z|)} \right)\\
&\le& \limsup_{|z| \to 1} \left(\frac{\log C}{\psi(|z|)} + \frac{\log \varphi_p(|z|)}{\psi(|z|)} \right)
= \limsup_{|z| \to 1} \frac{\log \varphi_p(|z|)}{\psi(|z|)} =  \alpha_p.
\end{eqnarray*}

Furthermore, by condition $\mathbf{(C_4)}$, since $\alpha_p < \alpha_{p+1}$, there exists some $\del_1 \in (0, 1)$, such that
\[
\frac{\log |f(z)|}{\psi(|z|)} < \frac{\alpha_p+\alpha_{p+1}}{2} < \frac{\log \varphi_{p+1}(|z|)}{\psi(|z|)},\ \text{for all $\del_1<|z|<1$}.
\]
From this it follows that
\[
\frac{|f(z)|}{\varphi_{p+1}(|z|)} < 1, \ \text{for all $\del_1<|z|<1$}.
\]

On the other hand, since $\varphi_{p+1}(w)\in(1,\infty)$, we also have
\[
\frac{|f(z)|}{\varphi_{p+1}(|z|)} \le \sup_{|z| \le \del_1} |f(z)| < \infty.
\]
So we get
\[
\sup_{z\in\B}\frac{|f(z)|}{\varphi_{p+1}(|z|)} < \infty,
\]
that is $f \in A_\varphi^{-m}$, with $m=p+1$.

\subsection{Proof of (2)}

Let $S$ be a weakly sufficient set for $A_\Phi^{-\infty}$. We show that $T_{f, \psi, S} = T_{f, \psi}$ for every $f\in A_{\Phi}^{-\infty}$.

Assume in contrary that there exists a function $f \in A_{\Phi}^{-\infty}$, such that
\[
T_{f, \psi, S} < T_{f, \psi}.
\]
Then we can take $d\in(0,1)$ sufficiently small so that $T_{f, \psi, s}<T_{f, \psi}-d$.

Furthermore, by $\mathbf{(C_4)} - \mathbf{( C_5)}$, there is some $p\in\N$ big enough, such that
\begin{equation}\label{eq300}
\alpha_p > \alpha-d\quad\text{and}\quad f \in A_\varphi^{-p}.
\end{equation}

Since $S$ is weakly sufficient for $A_\Phi^{-\infty}$, by Lemma \ref{lem01}, there exists some $m=m(p) \ge p$ and $C_p$, such that
\begin{equation}\label{eq301}
\|g\|_m \le C_p\|g\|_{p,S}, \ \text{for all $g \in A_\Phi^{-\infty}$}.
\end{equation}

We need the following result.

\begin{lem} \label{19lem01}
$T_{f, \psi} \le \alpha_p$.
\end{lem}

\begin{proof}
Since $f \in A_\varphi^{-p}$, there exists some  $C_0=C_0(p)>1$, such that for any $z \in \B$,
\[
|f(z)| \le C_0 \varphi_p(|z|),
\]
which implies that
\[
\frac{\log |f(z)|}{\log \varphi_p(|z|)} \le \frac{\log C_0}{\log \varphi_p(|z|)}+1.
\]
Furthermore, note that by $\mathbf{(C_3)}$ and the assumption $\lim\limits_{r \to 1} \psi(r)=\infty$, we have $\dlim_{r \to 1} \log \varphi_p(r)=\infty$. Then
\[
T_{f, \psi}=\limsup_{|z| \to 1} \frac{\log|f(z)|}{\psi(|z|)}
=
\limsup_{|z| \to 1}\left\{ \frac{\log|f(z)|}{\log \varphi_p(|z|)} \cdot \frac{\log \varphi_p(|z|)}{\psi(|z|)}\right\} \le \alpha_p.
\]
The lemma is proved.
\end{proof}

Note that Lemma \ref{19lem01} gives
\[
T_{f, \psi} \le \alpha_m <\alpha,
\]
because $m \ge p$ and $(\alpha_p)$ is strictly increasing to $\alpha$.

Next we can choose $x,y>0$ small enough, so that
\begin{equation} \label{eq302}
\alpha_m+2y < x+T_{f,\psi} < \alpha-y.
\end{equation}
Indeed, first since $\alpha_m <\alpha$, there is $y>0$ for which
\[
\alpha_m + 3y < \alpha \quad (\Longleftrightarrow\ \alpha_m+2y < \alpha-y).
\]
Then since $0 \le \alpha_m - T_{f,\psi} < \alpha_m+2y - T_{f,\psi} < \alpha-y - T_{f,\psi}$, any choice of $x\in(\alpha_m+2y - T_{f,\psi},\alpha-y - T_{f,\psi})$ gives \eqref{eq302}.

\medskip
Now by the definition of $T_{f,\psi,S}$ and $T_{f,\psi}$, there exists some $A>0$, such that
\begin{equation} \label{eq303}
|f(z)| \le A e^{(T_{f,\psi}+y)\psi(|z|)}, \ \text{for all $z \in \B$},
\end{equation}
and
\begin{equation} \label{eq304}
|f(z)| \le A e^{(T_{f,\psi}-d)\psi(|z|)}, \ \text{for all $z \in S$}.
\end{equation}

Also there exists a sequence $\{z_k\} \subset \B$ with $\dlim_{k\to\infty}|z_k|=1$, such that
\begin{equation} \label{eq305}
|f(z_k)| \ge e^{(T_{f,\psi}-y)\psi(|z_k|)}, \ \text{for every $k=1, 2, \ldots$}
\end{equation}

By condition $\mathbf{(C_6)}$ with $t=x$ and $\xi \in \B$, we can have a function $g_{\xi,x} \in R(\psi,\xi,x)$, satisfying
\begin{equation} \label{eq306}
|g_{\xi,x}(z)| \le M e^{x\psi(|z|)}, \ \text{for all $z \in \B$}.
\end{equation}

Consider the function $h_{\xi,x}=f\cdot g_{\xi,x} \in\calO(\B)$. We prove the following.

\begin{lem} \label{19lem02}
$h_{\xi,x} \in A_\Phi^{-\infty}$.
\end{lem}

\begin{proof}
From \eqref{eq303} and \eqref{eq306}, it follows that
\[
|h_{\xi,x}(z)| \le AM e^{(T_{f,\psi}+x+y)\psi(|z|)}, \ \text{for all $z \in \B$}.
\]
By \eqref{eq302}, we have $T_{f,\psi}+x+y<\alpha$, and hence by $\mathbf{(C_4)}$, there is $\widetilde{p}$ large enough, such that
\[
T_{f,\psi}+x+y < \alpha_{\widetilde{p}-1} < \alpha_{\widetilde{p}} < \alpha,
\]
which gives
\begin{equation}\label{eq307}
|h_{\xi, x}(z)| \le AM e^{\alpha_{\widetilde{p}-1} \psi(|z|)},\ \text{for all $z\in\B$}.
\end{equation}

Furthermore, by $\mathbf{(C_3)}$, for $0 < \eps' < \dfrac{1}{3}(\alpha_{\widetilde{p}}-\alpha_{\widetilde{p}-1})$, there exists $\del(\widetilde{p}) \in (0,1)$, such that when $\del(\widetilde{p})<|z|<1$,
\[
\alpha_{\widetilde{p}-1}- \eps' < \frac{\log \varphi_{\widetilde{p}-1}(|z|)}{\psi(|z|)} < \alpha_{\widetilde{p}-1}+\eps'
\]
and
\[
\alpha_{\widetilde{p}}- \eps' < \frac{\log \varphi_{\widetilde{p}}(|z|)}{\psi(|z|)} < \alpha_{\widetilde{p}}+\eps',
\]
or equivalently,
\[
\varphi_{\widetilde{p}-1}(|z|)e^{-\eps' \psi(|z|)} < e^{\alpha_{\widetilde{p}-1} \psi(|z|)} < \varphi_{\widetilde{p}-1}(|z|)e^{\eps' \psi(|z|)}
\]
and
\[
\varphi_{\widetilde{p}}(|z|)e^{-\eps' \psi(|z|)}  < e^{\alpha_{\widetilde{p}} \psi(|z|)} < \varphi_{\widetilde{p}}(|z|)e^{\eps' \psi(|z|)}.
\]
Hence, when $\del(\widetilde{p})<|z|<1$, we have
\begin{eqnarray*}
e^{\alpha_{\widetilde{p}-1} \psi(|z|)}%
&<& \varphi_{\widetilde{p}-1}(|z|) e^{\eps' \psi(|z|)} < e^{ \left(\alpha_{\widetilde{p}-1}+2\eps' \right) \psi(|z|)} \\
&<& e^{\left(\alpha_{\widetilde{p}}-\eps' \right) \psi(|z|)} < \varphi_{\widetilde{p}}(|z|).
\end{eqnarray*}
Thus for $\del(\widetilde{p})<|z|<1$,
\[
|h_{\xi, x}(z)| \le AM e^{\alpha_{\widetilde{p}-1} \psi(|z|)} \le AM \varphi_{\widetilde{p}}(|z|).
\]

On the other hand, since $\psi$ is continuous on $[0,\del(\widetilde{p})]$, there exists a positive constant $\widetilde{C}$ (depending only on $\widetilde{p}$ and $\alpha_{\widetilde{p}-1}$), such that
\[
\sup_{|z| \le \del(\widetilde{p})} e^{\alpha_{\widetilde{p}-1} \psi(|z|)} \le \widetilde{C}.
\]

Combining the last two inequalities, we conclude that for some $\widetilde{A}>0$, depending only on $A$, $M$, $\widetilde{p}$ and and $\alpha_{\widetilde{p}-1}$,
\begin{equation}\label{eq308}
|h_{\xi, x}(z)| \le \widetilde{C} \varphi_{\widetilde{p}}(|z|),\ \text{for all $z\in\B$}.
\end{equation}
This shows that $h_{\xi, x} \in A^{-\widetilde{p}}_\varphi$, and hence $h_{\xi, x} \in A^{-\infty}_{\Phi}$.  The lemma is proved.
\end{proof}

Now we choose and fix $\eps_0\in(0,d)$. In this case, since
\[
T_{f,\psi}-d+x < \alpha-y-d < \alpha-d < \alpha_p-\eps_0
\]
(here $p$ is given in \eqref{eq300}), by a similar way as the proof of \eqref{eq308}, due to conditions \eqref{eq304} and \eqref{eq306}, there exists $B>0$ such that for any $z \in S$,
\begin{equation} \label{eq309}
|h_{\xi,x}(z)| \le AM e^{(T_{f,\psi}-d+x)\psi(|z|)} \le AM e^{\left(\alpha_p-\eps_0\right)\psi(|z|)}  \le B\varphi_p(|z|).
\end{equation}
This shows that $f\cdot g_{\xi,x} \in A_\varphi^{-p,S}$. Here, the constant $B$ depends on $A$, $M$, $p$ and $\eps_0$.  In particular, it is independent of the choice of $\xi \in \B$.

Similarly, applying \eqref{eq301} to $h_{\xi,x}$, by \eqref{eq309} and Lemma \ref{19lem02}, we see that there exists $D>0$, such that
\begin{equation} \label{eq3010}
|h_{\xi,x}(z)| \le C_pB \varphi_m(|z|) \le D e^{(\alpha_m+\del_0)\psi(|z|)}, \ \text{for all $z \in \B$}.
\end{equation}
Here, $\del_0>0$ is sufficiently small for which
\begin{equation} \label{eq314}
\alpha_m+2y+\del_0<x+T_{f,\psi},
\end{equation}
while the constant $D$ only depends on $C_p$, $B$, $\del_0$, boundedness of $\varphi_1$ on compact subsets of $[0,1)$, and on conditions $\mathbf{(C_2)}-\mathbf{(C_3)}$.In particular, it is independent of the choice of $\xi \in \B$.

For each $k \ge 1$, we let $z=\xi=z_k$, where $(z_k)$ are taken from \eqref{eq305}. Then  since $g_{z_k,x} \in R(\psi,z_k,x)$, we have
\[
|g_{z_k,x}(z_k)|=e^{x\psi(|z_k|)}.
\]
Putting $g_{z_k,x}(z_k)$ into \eqref{eq308} and taking into account \eqref{eq314}, we have
\[
|f(z_k)| = \left|\frac{h_{z_k,x}(z_k)}{g_{z_k,x}(z_k)}\right| \le D e^{(\alpha_m-x+\eps)\psi(|z_k|)} \le D e^{(T_{f,\psi}-2y)\psi(|z_k|)},
\]
which contradicts to \eqref{eq305} if $k$ is sufficiently large.

The theorem is proved completely.

\medskip
As  a consequence of Theorem \ref{thm01} and Theorem \ref{thm02}, we have the following result.

\begin{cor}
Let $S$ be a subset of $\B$. Let further, $\Phi$ and $\psi$ satisfy the conditions $\mathbf{(C_1)} - \mathbf{(C_6)}$. If in addition, the weight $\varphi_1$ is bounded on any compact subset of $[0,1)$, then the following assertions are equivalent:
\begin{enumerate}
\item[(i)] $S$ is $\psi$-sampling for $A^{-\infty}_\Phi$ relative to $\psi$.
\item[(ii)] $S$ is weakly sufficient for $A^{-\infty}_\Phi$.
\item[(iii)] $S$ is a set of uniqueness for $A^{-\infty}_\Phi$.
\end{enumerate}
\end{cor}

\begin{proof}
It suffices to recall that the condition $\mathbf{(C_2)}$ is guaranteed if $\Phi$ consists of continuous weights.
\end{proof}

\begin{exa}
To complete our exposition, we provide some examples of weights $(\varphi_p)$ and $\psi$ that satisfy all conditions $\mathbf{(C_1)} - \mathbf{(C_6)}$ in the main results of our paper. Clearly, there is a lot of such weights.
\begin{enumerate}
\item [(1)] $\Phi=\left(\left(1-r^2 \right)^{-2+\frac{1}{p}} \right)_{p=1}^\infty$ and $\psi(r)=-\log (1-r^2)$.
\item [(2)] $\Phi=\left(\left(1-e^{r-1} \right)^{-3+\frac{4}{p^2}} \right)_{p=1}^\infty$ and $\psi(r)=-\frac{\log (1-r)}{10}$.
\item [(3)] $\Phi=\left(\sin \left(\frac{1-r}{100} \right)^{-10+\textup{arccot}(p)} \right)_{p=1}^\infty$ and $\psi(r)=-5\log(1-r^2)$.
\end{enumerate}
\end{exa}

\bigskip

\end{document}